\documentclass[12pt,final]{scrartcl}
\usepackage[english]{babel}
\usepackage{header}
\usepackage{fixme}                     
\input xy
\xyoption{all}

\newcommand{\po}{\ar@{}[dr]|{\text{\pigpenfont R}}}
\newcommand{\pb}{\ar@{}[dr]|{\text{\pigpenfont J}}}

\author{Tarje Bargheer}
\title{String Topology for Embedding Spaces}

\begin{document}
\maketitle
\section{Introduction}\label{intro}

Since its inception \cite{chassullivan}, string topology has been concerned with homotopical and algebraic structures associatied to the free loop space, $\Map(S^1,M)$, and sporadically higher dimensional variants of this \cite{PoHuMapping} \cite{GinotTradlerZeinalian}. The aim of this paper is to extend string topology to embedding spaces, also encompassing embedding spaces of higher dimensional spheres. This will be given by operadic maps of parametrised spectra

$$
\Fun(\coprod^k S^n,M) \times \Cleav_{S^n}(-;k) \to \colim_{C(\Gamma)}\Fun(S^n,M)^{TM(C(\Gamma))}.
$$

The exact nature of this morphism of parametrised spectra is presented in theorem \ref{globalmap}. The notation $\Emb(S^n,M)^{TM(C(\Gamma))}$ indicates that the colimit is taken over a sequence of spaces that up to homotopy are Thom spaces over the embedding space $\Fun(S^n,M)$ of embeddings from $S^n$ into $M$. The theorem hence provides a spectrum level version of higher dimensional string topology for embedding spaces. This foundational idea of having string topology presented through a spectra first appeared in \cite{homotopystring}. As we also specify in \ref{globalmap}, the above map of parametrised spectra leads to an action map on homology

$$
\h_p(\Fun(\coprod^k S^n,M)) \otimes \h_q(\Cleav_{S^n}(-;k)) \to \h_{p+q-\dim(M)(k-1)}(\Fun(S^n,M))
$$

\begin{figure}[h!tb]\label{almostintersectionfigure}
\includegraphics[scale=.8]{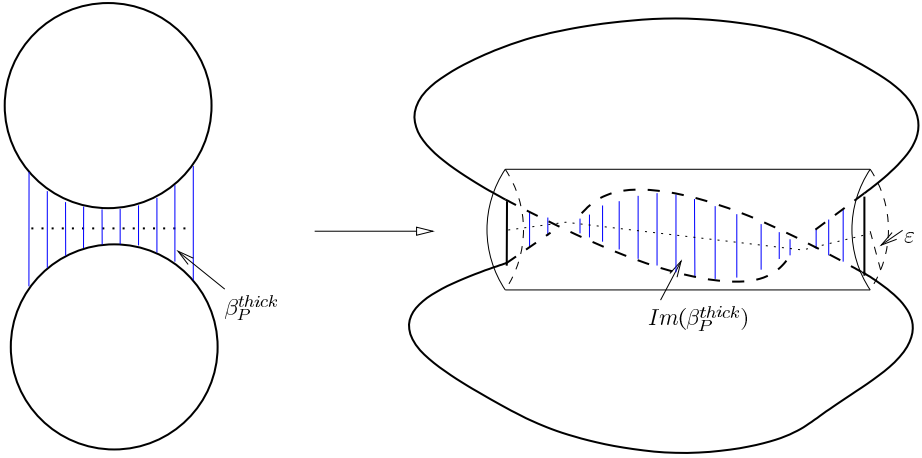}
\caption{On the left, two copies of $S^1$, connected by thin lines, representing a thickened blueprint $\beta^{\thick}_{P}$. As long as the connected portion of the embeddings of these are within a $\varepsilon$-neighborhood of each other, and one still obtains an embedding, they are connected to a single embedding, represented by the perimeter of the image on the right. The other cases are handled continously under the umkehr map we provide by the Thom collapse.}
\end{figure}

Roughly speaking, the Thom spaces in the spectral map are used to measure the proximity of embeddings along specified portions, called the blueprint, of the sphere in the domain. These Thom spaces are modelled over the tangent bundle $TM$ to assign the difference between two close points inside $M$. If they are far away from each other they are mapped to a point at $\infty$. This leans on the usual construction of Thom collapse maps for manifolds, however more complexity is added since we have a tangent bundle over every point of the domain of the embedding, and we furthermore need to account for several embeddings. We do this by exploiting the core intricacies of the cleavage operad as developed in \cite{TarjeCleavage}.

As promoted by for instance \cite[2.1]{LeinsterHigher}, symmetric monoidal categories are additional data for an action of an operad, and the action we provide of the cleavage operad is indeed not a symmetric monoidal category, meaning that we need to use the functor $\Fun(\coprod^k -,*)$. Had we $\Map$ instead of $\Fun$, the adjunction $\Map(\coprod^k -,*) \cong \prod^k\Map(-,*)$ would yield an action of a symmetric monoidal category.

While we do employ homotopy theoretic techniques, embedding spaces are inherintly more geometric. This also entails that our construction of umkehr maps, through a Thom collapse map, have a geometric nature to them. Umkehr maps in string topology has been considered as a consequence of Poincar\' e duality in for instance \cite{umkehrCohenKlein}. This is not the case for the geometry we utilise, and in principle, our construction does not need to assume that $M$ is compact. However, in order to prove that the associated structure is invariant under the choice of metric on $M$, we rely on compactness of $M$ \ref{shriekproperties}[D]. 

We use the final portion of the paper to show how string topology for embedding spaces and string topology for mapping spaces relate to each other. We do this through a homotopy commutative diagram of parametrised spectra:

$$
\xymatrix{
\Fun(\coprod^k S^n,M) \times \Cleav_{S^n}(-;k) \ar[r]\ar[d] & \colim_{C(\Gamma)}\Fun(S^n,M)^{TM(C(\Gamma))}\ar[d]\\
\prod^k\Map(S^n,M) \times \Cleav_{S^n}(-;k) \ar[r] & \colim_{C(\Gamma)}\Map(S^n,M)^{\widehat{TM(C(\Gamma))}}
}
$$

with the homotopy defect of the diagram given explicitly in the diagram (\ref{homotopydefect}). The spaces $\Map(S^n,M)^{\widehat{TM(C(\Gamma))}}$ in the lower-right corner of the square are a priori more complicated objects than the Thom spaces in the upper-right corner of the square. The core reason behind these more complicated spaces can be seen by considering the figure \ref{almostintersectionfigure} -- and allowing the two embeddings to start intersecting, as they are able to in mapping spaces.

While the proximity of two curves that intersect at $x$ are naturally $0$ at $x$, one can pointwise in $\prod^k\Map(S^n,M)$ do the same construction of the umkehr map as for embedding spaces. The problem with $0$ is however that it does not scale to anything but $0$ -- and this leads to continuity problems for the umkehr map. Therefore, $\Map(S^n,M)^{\widehat{TM(C(\Gamma))}}$ is given by disregarding these points of self-intersection. 

A priori, the above changes the homotopy type of the spaces. While we do not investigate it in this paper, we imagine that these spaces become increasingly difficult to handle as the dimension of the $S^n$ in the domain grows.

However, for the $1$-dimensional case, the subspaces of self-intersection are given by subsets of line-segments -- which are contractible as long as the maps in question are smooth. This provides the key for showing that $\Map(S^1,M)^{\widehat{TM(C(\Gamma))}}$ are homotopy equivalent to the Thom spaces one usually works with in string topology, as presented in \ref{dimension1nottoobad}.

There are a couple of homotopies along the outline described above. Homology is a homotopy invariant, and while they are very explicitly described on the level of parametrised specte, we can disregard them on homology -- which gives the morphism of BV-algebras $$\hh_*(\Emb(\coprod^- S^1,M)) \to \hh_*(\Map(\coprod^- S^1,M)$$ as described in \ref{bvalgebramor}. This morphism can be considered as inducing the string topology structure along from the inclusion $\Emb(S^1,M) \to \Map(S^1,M)$. 

We do not show anything specific in this paper, but it appears that as $n$ grows, the space $\Map(S^n,M)^{\widehat{TM(C(\Gamma))}}$ grows in complexity compared to the homotopy Thom space $\Map(S^n,M)^{TM(C(\Gamma))}$. We find it curious what structure one can precure on these more complicated spaces.

The ideas in this paper started to take form from discussions with Craig Westerland about the interplay between mapping- and embedding-spaces. Hopefully our discussions can be taken a few steps further from here. We are also grateful to Haynes Miller for his nudges to give a more geometric version of the action in string topology, and to Nathalie Wahl for several helpful discussions.




\section{The Cleavage Operad}
This section is an outline of the constructions and ideas surrounding the cleavage operad $\Cleav_{S^n}$. The details are given in \cite[Ch. 3]{TarjeCleavage}. A $k$-ary element $[T,\underline{P}] \in \Cleav_{S^n}(-;k)$ is prescribed by a binary, rooted planar tree, $T$, with an ordering of the leaves, and the internal knots decorated by $k-1$ affine oriented hyperplanes $\underline{P}$, i.e. an element of the $n$-plane $(n+1)$-Grassmanian together with a real number used for translation away from $0 \in \RR^{n+1}$, $\left(\Gr_n(\RR^{n+1}) \times \RR\right)^{k-1}$. This data is subjugated to cleaving conditions.

Conceptually, these cleaving conditions are a formal way of prescribing the recursive procedure of cleaving an object into $k$ pieces. That is, the hyperplane $P$ decorated on the internal knot closest to the root of $T$ are required to cleave $S^n$ into two closed submanifolds $V_1,V_2 \subset S^n$, representing outgoing colours of the operadic element. That is, $V_1 \cup V_2 = \left(\RR^{n+1} \setminus P\right) \cap S^n$ and $V_1 \cap V_2 \subset P$. The orientation of $P$ determines $V_1$ as the component in the direction of the normal-vector of $P$. The left-most of the branches of $T$ recursively prescribe how $V_1$ -- in place of $S^n$ -- is cleaved by the further hyperplane decorations, and the right-most how to cleave $V_2$. As further subsets of $V_1$ or $V_2$, the $k$-ary operation will hence have associated $k$ outgoing colours $N_1,\ldots,N_k \subset S^n$.

We finally apply a quotient that identifies any two tuples $(T,\underline{P})$ and $(T',\underline{P}')$ if they result in the same outgoing colours $N_1,\ldots,N_k$. Each of these hence represent the same element $[T,\underline{P}] \in \Cleav_{S^n}(-;k)$.

\begin{figure}[h!tb]
\includegraphics[scale=1]{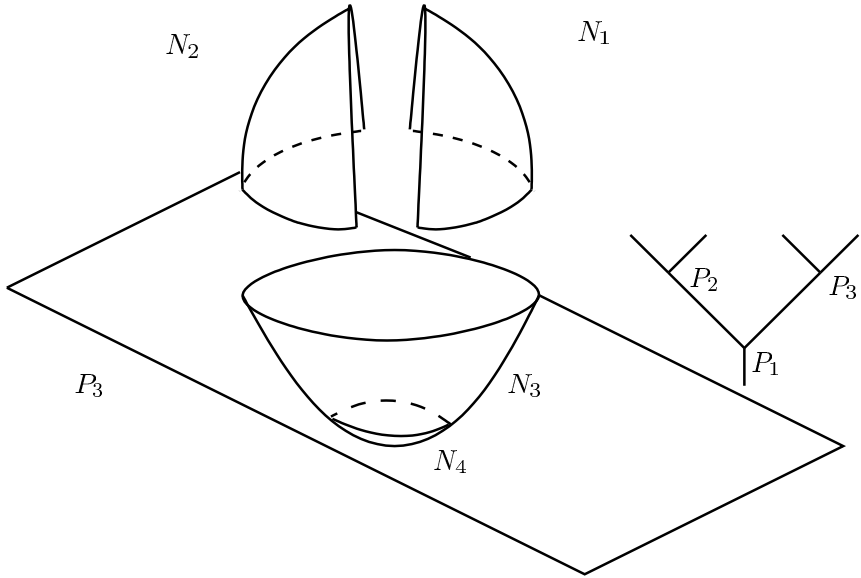}
\label{spherecleavepic}
\caption{A step in the recursive procedure for the cleavage operad of a $2$-sphere. The hyperplane $P_3$ cleaves the lower portion of the sphere into two parts $N_3$ and $N_4$, while the tree decorated by $P_1,P_2,P_3$ indicates that prior to this cleavage the sphere was first cut in two hemispheres by $P_1$, whereafter $P_2$ cut the upper hemisphere in yet two parts: $N_1$ and $N_2$}
\end{figure}

Since the unit disk $D^{n+1} \subset \RR^{n+1}$ is convex, the hyperplanes defining a cleavage of the unit disk $S^n \subset \RR^{n+1}$ can be extended to a cleavage of $D^{n+1}.$ This gives an embedding of operads $$U \colon \Cleav_{S^n} \to \Cleav_{D^{n+1}}.$$ 

The \emph{blueprint} is the subset $\beta_{[T,\underline{P}]} \subset D^{n+1}$ given by 
$$\beta_{[T,\underline{P}]} := \bigcup \partial U(N_i).$$ 
That is, the boundary $\partial U(N_i)$ is given by a subset of the union of the hyperplanes $\underline{P}$, and in this sense the blueprint is formed by the points of the hyperplanes that contribute to the cleaving process.

Dictated from the topology on the space of hyperplanes that bound the outgoing colour as subsets of $\RR^{n+1}$, there is a topology on the set of all outgoing colours, $\Ob(\Cleav_{S^n})$. This is defined in \cite[Ch 3.2]{TarjeCleavage}. An equivalence relation, chop equivalence, is imposed - equivalating two recursive cleaving procedures that yield the same outgoing colours. This specifies $\Cleav_{S^n}$ as a coloured topological operad \cite[Def. 2.6]{TarjeCleavage}. The space of outgoing colours have an action by $\Sigma_k$ permuting the ordering of the leaves on the indexing trees. In the above we have only specified $S^n$ as incoming colour; however any element of $\Ob(\Cleav_{S^n})$ can be obtained as incoming by taking a decorated sub-tree of $S \subset T$. The recursive procedure specified by the entire tree will give $W \in \Timber_{S^n}$ that $S$ is prescribed to cleave.

The main result of \cite[Th. 5.21]{TarjeCleavage} is that $\Cleav_{S^n}$ is a coloured $E_{n+1}$-operad. Since $\Timber_{S^n}$ is shown to be contractible \cite[Th. 3.17]{TarjeCleavage}, this means that under homotopy invariants actions of $\Cleav_{S^n}$ is precisely $E_{n+1}$-algebras.

To account for incoming colours, our notation is $\Cleav_{S^n}(V;k)$ for the space of $k$-ary cleaving operations with incoming colour $V$ and $k$ varying outgoing colours. The larger space where the input colours is also varying over $\Ob(\Cleav_{S^n})$ is denoted $\Cleav_{S^n}(-;k)$.

\section{Thickening of the Blueprint}

Using the inclusion $U \colon \Cleav_{S^n} \to \Cleav_{D^{n+1}}$, to any timber $N_i \subset S^n$, the corresponding $U(N_i)$ will have a well-defined centre of mass $c_i \in D^{n+1}$. To every point $s \in S^n$, the line from this point to $c_i$ will be denoted $l(s,c_i)$. As a geometric seed for our construction of the string topology action, we define a map

\begin{eqnarray}\label{alphamap}
\alpha_{[T,\underline{P}]} \colon \coprod_{i=1}^k \complement N_i \to \beta_{[T,\underline{P}]}
\end{eqnarray}

This is given by mapping $s \in \complement N_i$ to the point in $\beta_{[T,\underline{P}]}$ that lies on $l(s,c_i)$ and is on the boundary of $N_i$.

\begin{figure}[h!tb]
\includegraphics[scale=1]{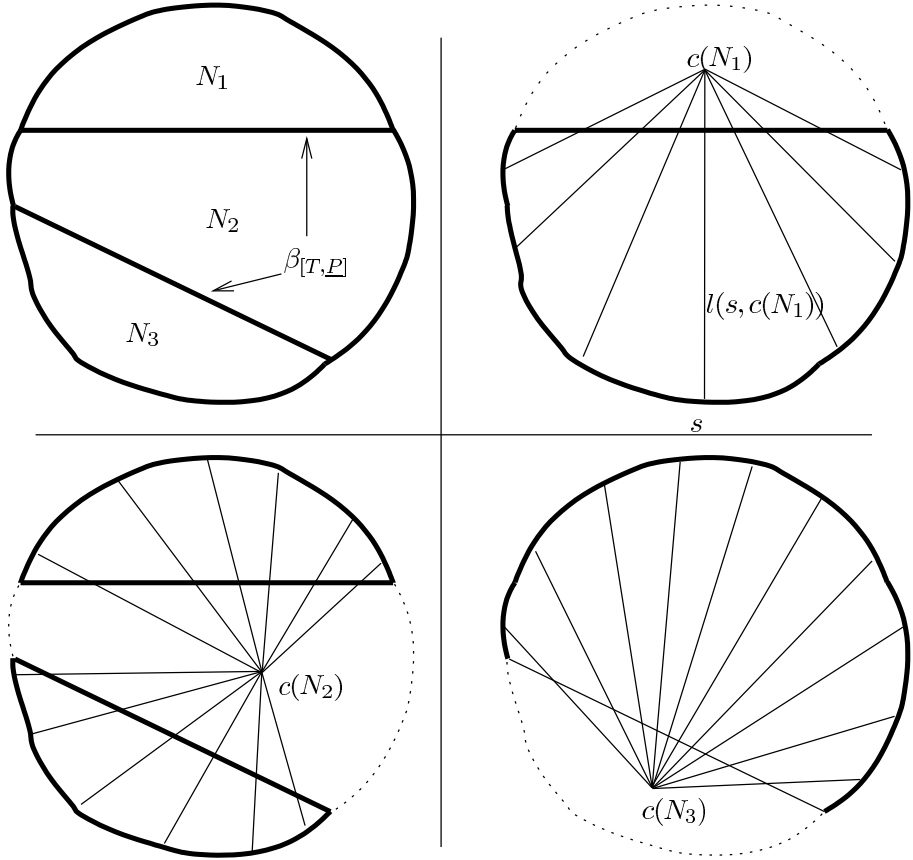}
\label{alphamappic}
\caption{The top left picture shows a cleavage of the disk. In the other three, where the lines intersect the blueprint is where the boundary of the complement of the outgoing is mapped to}
\end{figure}

\begin{proposition}
The map $\alpha_{[T,\underline{P}]}$ is well defined and injective when restricted to a single component of one of the complements $\complement N_i$.
\end{proposition}
\begin{proof}
The outgoing colour $N_i$ is convex as it can be seen as the intersection $\bigcap_{j=1}^r D^{n+1}_{P_{j}}$ where $\{P_{j}\}$ is the set of hyperplanes bounding $N_i$ and $D^{n+1}_{P_j}$ is the portion of the disk cleaved in two by $P_j$ that has $D^{n+1}_{P_j} \cap (N_i)^{\circ} \neq \emptyset$. Since all $D^{n+1}_{P_j}$ are convex $N_i$ is also convex.

This means that $c_i \in N_i$ will be a point in the interior of $N_i$, and the line from $c_i$ to a point of the boundary $\partial N_i$ is uniquely defined. Extending this line will eventually hit any point of $\complement N_i$, since the line is unique, this shows that the map $\alpha_{[T,\underline{P}]}$ is injective restricted to $\complement N_i$. 
\end{proof}

In order to work with spaces of embeddings, it is essential to make the map $\alpha_{[T,\underline{P}]}$ injective as a map from the entire $\coprod_{i=1}^k \complement N_i$, not just on the components. This spurs a homotopical replacement $\beta_{[T,\underline{P}]}^{\thick}$ of the blueprint $\beta_{[T,\underline{P}]}$: 

Notice that $|\alpha_{[T,\underline{P}]}^{-1}(b)| = p+1$ for a point $b \in \beta_{[T,\underline{P}]}$ that sits where $p$ hyperplanes come together in $\beta_{[T,\underline{P}]}$. A single hyperplane will have two outgoing colours at either side of it, and each additional hyperplane will add another outgoing colour whose complement hits the point $b$ under the map $\alpha_{[T,\underline{P}]}$. 

To the $p$-simplex $\Delta^p$, we let the \emph{$i$-spine} $\spine^p[i]$ be the union of the $1$-simplices that has the $i$'th vertex of $\Delta^p$ as one of its vertices. Alternatively, $\spine^p[i]$ is the spine of the $i$'th horn $\Lambda^p_i \subset \Delta^p$. 

The homotopical replacement $\beta_{[T,\underline{P}]}^{\thick}$ will be given by replacing $b \in \beta_{[T,\underline{P}]}$ with $\coprod_{i=1}^{p+1} \spine^p[i]$ when $|\alpha^{-1}_{[T,\underline{P}]}(b)| = p+1$. 

The map $\alpha_{[T,\underline{P}]}$ maps $\complement N_i$ injectively into subset of $\beta_{[T,\underline{P}]}$. A point $b = \alpha_{[T,\underline{P}]}(a_{i_1}) = \cdots = \alpha_{[T,\underline{P}]}(a_{i_{p+1}})$ where $a_{i_{j}} \in \complement N_{i_j}$ and $i_1 < \cdots < i_{p+1}$ specifies for every $j \in \{1,\ldots,p+1\}$ the $j$-spine $\spine^p[j]$ at $b \in \beta_{[T,\underline{P}]}$.

We define the \emph{thickened blueprint} as a subset  
\begin{eqnarray}\label{thickblueprint}
\beta_{[T,\underline{P}]}^{\thick} \subset \coprod_{i=1}^k \beta_{[T,\underline{P}]} \times \spine^{k-1}[i].
\end{eqnarray}

This subset is obtained by noting that $p_b := |\alpha^{-1}_{[T,\underline{P}]}(b)| \leq k$ where $k$ is the arity of $[T,\underline{P}]$. The $i$-spine at every $b \in \beta_{[T,\underline{P}]}$ specifies a subset $\spine^{p_b} \subset \spine^{k-1}[i]$ and these subsets varying with $b$ specify the inclusion (\ref{thickblueprint}). 

An effect of the thickened blueprint is that we have an injective map

$$
\alpha_{[T,\underline{P}]}^{\thick} \colon \coprod_{i=1}^k \complement N_i \to \beta_{[T,\underline{P}]}^{\thick} 
$$

where for $x_i  \in \complement N_i$, the map is given by $\alpha_{[T,\underline{P}]}^{\thick}(x_i) = \left(\alpha_{[T,\underline{P}]}(x_i),\spine^{p_{\alpha_{[T,\underline{P}]}(x_i)}}\right)$ 


\section{Colliding and Evading Hyperplanes}\label{collidingandevadingsection} 

Let $\hat{C}(\Gamma) \subset \Cleav_{S^n}(-;k)$ denote the subset where $|\pi_0(\beta_{[T,\underline{P}]})| = \Gamma$.

We remark that $\hat{C}(0)$ is given by configurations of hyperplanes that form a single component $D^{n+1}$. At the other extreme, $\hat{C}(k-1) \simeq \Sigma_{k+1}$, in the sense that the space of hyperplanes not intersecting within $D^{n+1}$ has contractibly components, one for each labelling of the outgoing colours -- determined by $\Sigma_{k+1}$.

Since $C(\Gamma+1)$ is given by moving a hyperplane from $C(\Gamma)$ into a coulour of $\Cleav_{D^{n+1}}$, a convex subset of $D^{n+1}$, it follows that $C(\Gamma)$ has contractible components.

Consider the collision space $$C(\Gamma,\Gamma') := \hat{C}(\Gamma) \cap \overline{\hat{C}(\Gamma')}.$$ 

The closure is taken within $\Cleav_{S^n}$, meaning that the limit points in the closure will be where the amount of components of the blueprint decreases.

We let $C(\Gamma)$ be homeomorphic to the space $\overline{\hat{C}(\Gamma)}$ but let points $[T,\underline{P}]$ of $\overline{\hat{C}(\Gamma)} \setminus \hat{C}(\Gamma)$ be equipped with an \emph{evasive blueprint} $\beta_{\evade{[T,\underline{P}]}}^{\thick}$. We define the evasive blueprint as the limit of blueprints in $\hat{C}(\Gamma)$, meaning that for $[T,\underline{P}] \in \overline{\hat{C}(\Gamma)} \setminus \hat{C}(\Gamma)$, the blueprint will have components joined at a single point, where the corresponding components are separate in $\hat{C}(\Gamma)$, and a spine of a larger simplex $\Delta_x$ at the joined point. The limit inside $\hat{C}(\Gamma)$ will hence retain the amount of components of $\beta_{[T,\underline{P}]}$, and have the spine of a subsimplex of $\Delta_x$ over the points that are not joined in the limit.

We perform an iterated glueing of these spaces via a direct limit construction using correspondences  

\begin{eqnarray}\label{cleavingpushout}
\xymatrix{C(\Gamma) & C(\Gamma,\Gamma') \ar[l]^{\phi_{\in}} \ar[r]^{\phi_{\evade}} & C(\Gamma')}
\end{eqnarray}

as the basic building blocks. Here the maps $\phi_{\in}$ are inclusions into the boundary with $\phi_{\in}$ mapping to the standard thickened blueprint, and $\phi_{\evade}$ the evasive blueprint. Pullbacks of these correspondences form a category $\mathfrak{C}$ with the pullback space formed from $C(\Gamma),C(\Gamma')$ and $C(\Gamma'')$ being $C(\Gamma,\Gamma'')$. 

\begin{proposition}\label{colimitcleavage} 
$$
\colim_{\mathfrak{C}}\colim(\xymatrix{C(\Gamma) & C(\Gamma,\Gamma')\ar[l]\ar[r] & C(\Gamma')}) \cong \Cleav_{S^n}(-;k)
$$
\end{proposition} 
\begin{proof}
The inner colimit glues the subspace $C(\Gamma)$ and $C(\Gamma')$ together along the boundary prescribed by $C(\Gamma,\Gamma')$. Any $[T,\underline{P}]$ with $|\pi_0(\beta_{[T,\underline{P}]}| = \Gamma$ is in $C(\Gamma)$, and since $C(\Gamma,\Gamma')$ uniquely determines any limit of blueprints with different amounts of components, the outer colimit glue to a space homeomorphic to $\Cleav_{S^n}(-;k)$
\end{proof}

Note that in the colimit, there is a discrepancy as to what blueprint we assign to a cleavage; whether it is the evasive or ordinary blueprint. This will lead us to produce a stable string topology action of the entire $\Cleav_{S^n}(-;k)$ - whereas $C(\Gamma)$ can be made to act unstably, up to a fixed amount of suspensions dependent on $\Gamma$.

\section{Homotopy Thom Spaces}

To describe mapping spaces over any subspace $C \subseteq \Cleav_{S^n}(-;k)$, we do the following:
 
The disjoint union $\coprod_{[T,\underline{P}] \in C} \Map(\coprod_{i=1}^k \complement N_i,M)$ receives a map from $\Map(\coprod^k S^n,M) \times C$ given by $(f_1,\ldots,f_k,[T,\underline{P}]) \mapsto (\res(f_1,\ldots,f_k))_{[T,\underline{P}]}$. Let a basis for the target be given by the image of a basis for the domain. 

For $[T,\underline{P}] \in C$, we describe an inclusion $\beta^{\thick}_{[T,\underline{P}]} \hookrightarrow \left(\prod^{k-1} D^{n+1} \right)\times \left(\coprod^k \Delta^{k}\right)$ by noting that $\beta_{[T,\underline{P}]}$ will consist of at most $k-1$ components, and each of these components have a inclusion into $D^{n+1}$ along with a deformation retraction from $D^{n+1}$ onto that component. This describes an inclusion of $\beta_{[T,\underline{P}]}$ into $\prod^{k-1} D^{n+1}$, over each point in the thickened blueprint we have added at most $k$ disjoint spines as subsets of $\Delta^k$.

As an effect, the inclusion provides a surjective map 
$$\Map\left(\left(\prod^{k-1} D^{n+1}\right) \times \left(\coprod^k \Delta^k\right),M\right) \times C \to \coprod_{[T,\underline{P}] \in C} \Map(\beta_{[T,\underline{P}]},M).$$

We give the target of this map the quotient topology and denote it $\Map(\beta^{\thick},M)_{C}$.

Let $p \colon TM \to M$ denote the tangent-bundle. This induces a map 

$$
(p_*)_{\Cleav_{S^n}(-;k)} \colon \Map(\beta^{\thick},TM)_{\Cleav_{S^n}(-;k)} \to \Map(\beta^{\thick},M)_{\Cleav_{S^n}(-;k)}
$$

The map is however not a fibration; noting that for a fixed $([T,\underline{P}]\in \Cleav_{S^n}(-;k)$ with  $f \in \Map(\beta^{\thick}_{[T,\underline{P}]},M)$ the fiber over this point will be $\Map(\beta_{[T,\underline{P}]},\RR^{\dim(M)})$; since the connected components of $\beta_{[T,\underline{P}]}$ vary with varying $[T,\underline{P}] \in \Cleav_{S^n}(-;k)$, so restricting to the maps into the unit-sphere of the domain $\RR^{\dim(M)}$ would have non-homotopy equivalent fibers.

Note however that restricting to $C(\Gamma) \subset \Cleav_{S^n}(-;k)$, the blueprints $\beta_{[T,\underline{P}]}$ are all homotopy equivalent for $[T,\underline{P}] \in C(\Gamma)$. Restricting to these subspaces, we obtain the following:

\begin{proposition}\label{localfibration}
The map $p_*|_{C(\Gamma)}$ is a fibration.
\end{proposition}
\begin{proof}
This follows by taking a fixed $[T,\underline{P}] \in C(\Gamma)$. Since $C(\Gamma)$ has contractible components we can form the pullback diagram
$$
\xymatrix{
\Map(\beta^{\thick},TM)_{C(\Gamma)}\ar[d]\ar[r]\pb & \Map(\beta^{\thick}_{[T,\underline{P}])},TM) \times C(\Gamma)\ar[d]\\
\Map(\beta^{\thick},M)_{C(\Gamma)}\ar[r] & \Map(\beta^{\thick}_{[T,\underline{P}]},M) \times C(\Gamma).
}
$$
The right-hand side is a fibration with fiber $\Fun(\beta^{thick},\RR^{\dim(M))}$. The desired morphism on the left side in the pullback is a fibration as well.
\end{proof}

The thickened blueprint over a point $b \in \beta_{[T,\underline{P}]}$ is given by a disjoint union $\coprod_{i=1}^{k} \spine^{p}[i]$, where $\spine^{p}[i] \subset \Delta^{k-1}$. We form the space $\beta_{[T,\underline{P}]}/\sim$ by identifying points that agree under the inclusion into $\Delta^{k-1}$. We let $\Fun_0(\beta^{\thick},M)_{C(\Gamma)} := \Fun(\beta^{thick}/\sim,M)_{C(\Gamma)}$.

Consider further the diagram of pullbacks, which also defines a sequence of relative embedding spaces -- basic for our homotopy Thom constructions:

$$
\xymatrix{
\Fun((S^n,\beta^{\thick})_{C(\Gamma)},(M,TM))\ar[r]\ar[d]^{\tilde{p}}\pb & \Fun_1(\beta^{\thick},TM)_{C(\Gamma)}\ar[d]\ar[r]\pb & \Map(\beta^{\thick},TM)_{C(\Gamma)}\ar[d]  \\ 
\Fun((S^n,\beta^{\thick})_{C(\Gamma)},(M,M)) \ar[r]\ar[d]\pb & \Fun_0(\beta^{\thick},M)_{C(\Gamma)} \ar[d]\ar[r] & \Map(\beta^{\thick},M)_{C(\Gamma)} \ar[dl]\\ 
\Fun(\coprod^k S^n, M) \times C(\Gamma)\ar[r] & \Map( \complement \out(C(\Gamma)),M)
}
$$

The space $\Fun_1(\beta^{\thick},TM)_{C(\Gamma)}$ is the space of maps from $\beta^{\thick}_{[T,\underline{P}]} \to TM$ such that when we compose with the projection map $p \colon TM \to M$ they are embeddings of $\beta^{\thick}_{[T,\underline{P}]}/\sim$. In particular, this implies that $\Fun_1(\beta^{\thick},TM)_{C(\Gamma)}$ are themselves embeddings.

The final pullback space $\Fun((S^n,\beta^{\thick})_{C(\Gamma)},(M,TM))$ for a fixed $[T,\underline{P}] \in C(\Gamma)$ consist of pairs $f_1 \colon S^n \to M$ and $f_2 \colon \beta^{\thick}_{[T,\underline{P}]} \to TM$ such that for $b \in \partial \beta^{\thick}_{[T,\underline{P}]}$ where $S^n \ni s = b$, we have $(p \circ f_2)(b) = f_1(s)$.

The top arrows in the diagram are fibrations by \ref{localfibration}. This allows us to construct a homotopical versions of the Thom space for $\Fun((S^n,\beta^{\thick})_{C(\Gamma)},(M,TM))$ in the following sense:

Let $CTM \subset TM$ denote the unit sphere bundle, where for $v \in T_pM$, $|v| \geq 1$. 
We let the homotopy Thom space be given by the quotient space

\begin{eqnarray*}
\Fun((S^n,\beta^{\thick}),M)^{TM(C(\Gamma))} :=\\
 \Fun((S^n,\beta^{\thick})_{C(\Gamma)},(M,TM))/\Fun((S^n,(\beta^{\thick}))_{C(\Gamma)},(M,CTM))
\end{eqnarray*}

This means that a function $f \in \Fun((S^n,\beta^{\thick})_{C(\Gamma)},M)^{TM(C(\Gamma))}$ can for $[T,\underline{P}] \in C(\Gamma)$ be considered as functions $f_1 \colon S^n \to M$ and $f_2  \colon \beta^{\thick}_{[T,\underline{P}]} \to TM$, under the same conditions as for $\Fun((S^n,\beta^{\thick})_{C(\Gamma)},(M,TM))$. These are subject to the further conditions that when $\exists b \in \beta^{\thick}_{[T,\underline{P}]}$ with $|f_2(b)| > 1$, we equivalate $f_2$ to a single $\infty$ along the domain of the entire component containing $b$.

\begin{proposition}\label{homotopythomiso}
The space $\Fun((S^n,\beta^{\thick}),M)^{TM(C(\Gamma))}$ is a homotopy Thom space, in the sense that choosing an orientation of $TM$ provides an isomorphism

$$
\h_{*+dim(M)}(\Fun((S^n,\beta^{\thick}),M)^{TM(C(\Gamma))}) \cong \h_*(\Fun((S^n,\beta^{\thick})_{C(\Gamma)},(M,M)))
$$
\end{proposition}
\begin{proof}
This follows as the standard Thom space argument, using the relative Serre Spectral sequence by considering the morphism of spectral sequences associated to the following morphism of fibrations, and functorially utilizing the $0$-section $M \to TM$ of $p \colon TM \to M$.

$$
\xymatrix{
\Fun((S^n,\beta^{\thick})_{C(\Gamma)},(M,CTM)) \ar[r]\ar[d] & \Fun((S^n,\beta^{\thick}),(M,TM))\ar[dl]\\
\Fun((S^n,\beta^{\thick})_{C(\Gamma)},(M,M)
}
$$
\end{proof}

There is a map $\Fun((S^n,\beta^{\thick})_{C(\Gamma)},(M,M)) \to \Fun(S^n,M)$ given by restricting $S^n \cup \beta^{\thick}_{[T,\underline{P}]}$ to $S^n$ for every $[T,\underline{P}] \in C(\Gamma)$. Using the Thom isomorphism above, on homology this induces a map $\h_{*+\dim(M)}(\Fun((S^n,\beta^{\thick}),M)^{TM(C(\Gamma))} \to \h_*(\Fun(S^n,M)$

\section{The Umkehr Map}

We shall describe an explicit Pointrjagin-Thom map into the homotopy Thom Space described in last section. 

$$
\alpha_{C(\Gamma)}^{\shriek} \colon \Fun\left(\coprod^k S^n,M\right) \times C(\Gamma) \to \Fun((S^n,\beta^{\thick}),M)^{TM(C(\Gamma))}
$$

We shall equip $M$ with a metric $g$, and assume that $M$ is geodesically complete. This choice is part of the data for the constructions of string topology for embedding spaces.

We write $\gamma \in \Fun(\coprod^k S^n,M)$ by its constituents $\gamma = \coprod_{i=1}^k \gamma_i$. As described in the previous section, an element $f \in \Fun((S^n,\beta^{\thick}),M)^{TM(C(\Gamma))}$ breaks for $[T,\underline{P}] \in C(\gamma)$ into two functions $f_1 \colon S^n\to M$ and $[f_2] \colon \beta^{\thick}_{[T,\underline{P}]} \to TM$ where $[f_2]$ is the equivalence class specified by the Thom construction. We specify $\alpha_{C(\Gamma}^{\shriek}$ by its value on these, denoted $\alpha_{C(\Gamma)}^{\shriek_1}$ for the value as $f_1$, and $\alpha^{\shriek_2}_{C(\Gamma)}$ for the value as $[f_2]$.

We let 

$$\alpha_{C(\Gamma)}^{\shriek_1}(\coprod_{i=1}^{k} \gamma_i)) = (\coprod_{i=1}^k \gamma_i|_{N_i})_{C(\Gamma)}.$$

Where $N_i$ is the $i$'th outgoing colour for $[T,\underline{P}] \in C(\Gamma)$.

To specify $\alpha^{\shriek_2}$, fix a Riemannian metric on $M$ and fix $\varepsilon > 0$, as well as a $[T,\underline{P}] \in C(\Gamma)$.

For a given point $b \in \beta_{[T,\underline{P}]}$, $\alpha^{-1}_{[T,\underline{P}]}(b) = \{x_1,\ldots,x_{p+1}\} \subset \coprod_{i=1}^k \complement N_i$. Assume that the geodesic distance satisfies $\dist(\gamma(x_i),\gamma(x_j)) < \varepsilon$. Having chosen $\varepsilon$ sufficiently small, and using geodesic completeness of $M$, there is a unique geodesic from $\gamma(x_i)$ to $\gamma(x_j)$. Denote this geodesic $g_{ij}^b$. The thickened blueprint over $b$, is given by $\coprod_{i=1}^p \spine[i]$. To obtain a map from this space into $TM$, we let the $1$-simplex to the $j$'th vertex of $\spine[i]$ parametrize at constant speed the geodesic $g_{ij}^b$. Let $\partial(g_{ij}^b)_t \in TM$ denote the tangent-vector of $g_{ij}^b$ at time $t$ and of length $1$.

Notice that after projecting with $p \colon TM \to M$ this indeed is a map from $\beta_{[T,\underline{P}]}^{\thick}/\sim$, since $g_{ij}^b(t) = g_{ji}^b(-t)$ -- and the latter is the geodesic parametrised by the $1$-simplex to the $i$'th vertex of $\spine[j]$.
 
Consider the neighborhood of $g_{ij}^b$ given by 

\begin{eqnarray}\label{tubularcigar}
N_{\varepsilon}(g_{ij}^b) := \bigcup_{t \in ]0,1[} D_{\varepsilon(\frac{1}{2} - |t - \frac{1}{2}|)}(g_{ij},t).
\end{eqnarray} 

Where $D_{r}(g_{ij}^b,t) \subset M$ is the disk centered at $p(g_{ij}^b(t)) \in M$, orthogonal to $g_{ij}^b$ at $t$ and of radius $r$.

If there are points $y \in \coprod^k S^n$ that has $\gamma(y) \in N_{\varepsilon}(g_{ij}^b)$. For each such $y$, there is a $t$ where $ \gamma(y) \in D_{\varepsilon(\frac{1}{2} - |t - \frac{1}{2}|}(g_{ij}^b,t)$. Let $\delta_y$ denote the distance inside this disk from $p(g_{ij}^b(t)$ to $\gamma(y)$, scaled such that $\delta_y = 1$ if $\gamma(y)$ is at the boundary of the disk. We let 

\begin{eqnarray}\label{scalingfactor}
S_{ij}^b := \frac{\dist(\gamma(x_i),\gamma(x_j))}{\varepsilon\inf_{y \in \coprod^k S^n} \delta_y}.
\end{eqnarray}

We let $S_{ij}^b = \infty$ when there exist $\gamma(y)$ with $p(g_{ij}^b(t)) = \gamma(y)$ for some $t \in ]0,1[$, or when $\dist(\gamma(x_i),\gamma(x_j)) > \varepsilon$. We define the second coordinate of the umkehr map by 

\begin{eqnarray}\label{umkehrformula}
\alpha_{C(\Gamma)}^{\shriek_2}(\gamma)(x_i) := \bigvee_{j \in \{1,\ldots,i-1,i+1,\ldots,p+1\}} \left(\sup_{b \in [T,\underline{P}]} S_{ij}^b\right) \cdot \partial(g_{ij}^b)
\end{eqnarray}

When $\dist(\gamma(x_i),\gamma(x_j)) \geq \varepsilon$, we assign the value at $\infty$ 

\begin{proposition}\label{shriekproperties}
$\alpha_{C(\Gamma)}^{\shriek}$ is continuous with the following properties: 
\begin{itemize}
\item[(A)] \underline{Thom-Embedding-soundness:} Let $\gamma(\in \Fun(\coprod^k S^n,M)$. If there exist an outgoing colour $N_l$ with $\gamma(N_l)$ intersecting the geodesic $\rho$ of length less than $\varepsilon$ between $\gamma(x) \in \gamma(\complement N_i)$ and $\gamma(y) \in \gamma(\complement N_j)$ for $\alpha_{[T,\underline{P}]}(x) = \alpha_{[T,\underline{P}]}(y)$ , then $\alpha_{[T,\underline{P}]}^{\shriek}(\gamma)$ will have the component of the domain containing $\alpha_{[T,\underline{P}]}(x)$ as the point at $\infty$.  
\item[(B)] \underline{Non-triviality:} Let two curves $\iota, \kappa \colon S^n \to M$ with $\iota \coprod \kappa \in \Fun(S^n \coprod S^n,M)$ have $\iota(N_1)$ and $\kappa(N_2)$ in an $\varepsilon$-neighborhood of each other, considered as subsets of $M$, where $N_1,N_2$ are the timber associated to a hyperplane $P$ cleaving $S^n$. Assuming (A) does not occur, $\alpha_{[T,\underline{P}]}^{\shriek}(\iota \coprod \kappa)$ will not have any points of its domain at the point at $\infty$. 
\item[(C)] \underline{Signed Symmetry:} For $U$ outgoing colour, and $[T,\underline{P}] \in \Cleav_{S^n}(U;2)$, let $\sigma \in \Sigma_2$ be the nontrivial permutation. $\sigma$ acts on $[T,\underline{P}]$ by permuting the elements, this leads to  

$$
\alpha_{\sigma.[T,\underline{P}]}^{\shriek}(\gamma) = -\alpha_{[T,\underline{P}]}(\gamma)
$$

where the minus sign is interpreted as the element is interpreted as taking $-x$ for all $x \in TM$ where $\alpha_{[T,\underline{P}]}(\gamma)$ takes values in $TM$.
\item[(D)] \underline{Metric Homotopy Invariance:} Assuming that $M$ is compact, the umkehr map is up to homotopy invariant of the choice of metric on $M$.

\end{itemize}
\end{proposition}
\begin{proof}
Continuity follows directly from the construction of the umkehr map.

(A) follows by the construction of the scaling factor $S_{ij}^b$, where $\delta_y$ measures if points of $g(N_l)$ are within $\varepsilon$ of intersecting $\gamma(\complement N_i)$ for some $i$, in which case $S_{ij}^b$ scales $\alpha_{[T,\underline{P}]}^{\shriek}(\gamma)$ by an increasingly large factor as points of $g(N_l)$ sits in an increasingly smaller $\delta$-neighborhood of $\gamma(\complement N_i)$ for $\delta < \varepsilon$.

(B) follows by the construction of the neighborhood $N_{\varepsilon}(g_{ij}^b)$, in the sense that portions of the domain of $\iota \coprod \kappa$ that are close to $\alpha_{[T,\underline{P}}^{-1}(\beta_{[T,\underline{P}})$, but do not start intersecting $\iota(N_1)$ or $\kappa(N_2)$, does not provide any additional scaling factor. This follows since the area of the neighborhood $N_{\varepsilon}(g_{ij}^b)$, we measure in tends to $0$ as one moves to the points of $\alpha_{[T,\underline{P}]}^{-1}(\beta_{[T,\underline{P}]})$. As there are no additional scaling factors in the assumption of (B), the umkehr map $\alpha_{[T,\underline{P}]}^{\shriek}(\iota \coprod \kappa)$ will not utilize any vectors of $TM$ with length less than $1$.  

(C) Acting with $\sigma$ on $[T,\underline{P}]$ means that we are interchanging the geodesic $g_{ij}$ with the geodesic $g_{ji}$. These two geodesics agree the same, only they go in opposite directions. This leads to opposing signs of $\partial(g_{ij}^b)$ and $\partial(g_{ji}^b)$ in the definition of $\alpha_{[T,\underline{P}]}^{\shriek}(\gamma)$, and the result follows.

(D) Let $\Tub(\beta_{[T,\underline{P}]}^{\thick},M)$ be the space of tubular neighborhoods of embeddings $\tau \in \Emb(\beta_{[T,\underline{P}]}^{\thick},M)$. Tubular neighborhoods in the sense that they satisfy (\ref{tubularcigar}) with $g_{ij}^b$ replaced with the embedding of the $1$-simplex associated to $\tau$ at the point $b \in \beta_{[T,\underline{P}]}$. Let $\Tub(\beta_{[T,\underline{P}]}^{\thick},M)^+$ be the one-point compactification. The choice of metric can for the sake of the construction of $\alpha_{[T,\underline{P}]}^{\shriek}$ be rephrased as a choice of tubular neighborhood $N \in \Tub(\beta_{[T,\underline{P}]},M)$, meaning that also the scaling factor $S_{ij}^b$ depends on $N$. Since we assume that $M$ is compact, we can use the proof of \cite[Prop. 31]{GodinHigher} to prove that $\Tub(\beta_{[T,\underline{P}]},M)$ deformation retracts onto $\Emb(\beta_{[T,\underline{P}]},M)$, so it follows that $\Tub(\beta_{[T,\underline{P}]},M)^+$ is homotopy equivalent to the homotopy Thom space $\Emb(\beta_{[T,\underline{P}]},M)^{TM(C(\Gamma))}$ and hence that up to homotopy $\alpha_{[T,\underline{P}]}^{\shriek}$ is independent of the choice of metric.
\end{proof}

\section{String Topology as Parametrised Spectra}

For a subspace $C(\Gamma) \subset \Cleav_{S^n}(-;k)$, the umkehr map in the previous section is a map from a mapping space to a homotopy Thom space constructed over a mapping space. However, taking the colimit over $\mathfrak{C}$ directly would not yield a map $\Fun(\coprod^k S^n,M) \to \colim_{C(r) \in \mathfrak{C}}\Fun(S^n,M)^{TM({C(r)})}$. Such a map would induce a map in homology; but the target is homotopy equivalent to a colimit over Thom space of varying degrees, varying the grading of the homology under Thom isomorphisms along with the varying components of the blueprint. This means one can only hope for a stable map over the entire cleavage operad.

To this end, consider the diagram  

\begin{eqnarray}\label{overdiagram}
\xymatrix{
C(\Gamma) \times \Fun(\coprod^k S^n,M) \ar[dr]\ar[rr]^{\alpha^{\shriek}_{C(\Gamma)}} & & \Fun((S^n,\beta^{\thick}),M)^{TM(C(\Gamma))}\ar[dl]\\
& \Fun((\coprod_{i=1}^k \out_i,M)_{C(\Gamma)},M)
}
\end{eqnarray}

Here, $\Fun((\coprod_{i=1}^k \out_i)_{C(\Gamma)},M)$ is the quotient space of $\coprod_{[T,\underline{P}] \in C(\Gamma)} \Fun((\coprod_{i=1}^k N_i,M))$ parametrised over $C(\Gamma)$ by the map from the top-left corner; meaning that for a given $[T,\underline{P}] \in C(\Gamma)$ is given by $\Fun(\coprod_{i=1}^k N_i,M)$ where $N_i$ is the $i$'th outgoing colour of $[T,\underline{P}]$.

Recall that for a retractive space $f \colon Y \to X$ with the retract $r \colon X \to Y$, the unreduced fiberwise suspension is given by the double quotient under $f$

$$
S_XY := X \cup_{(f,0)} Y \times [0,1] \cup_{(f,1)} X.
$$

Utilizing the retraction $r$, the reduced fiberwise suspension is

$$
\Sigma_X := S_XY \cup_{r \times [0,1]} S_XX
$$

Like an ordinary spectrum, a parametrised spectrum over $X$ is given by a sequence of retractive spaces $S(n)$ index by $n \in \ZZ$ and maps

$$
\Sigma_X S(n) \to S(n+1)
$$

The diagram (\ref{overdiagram}) is not a diagram of retractive spaces, and hence do not fit into the theory of fibered spectra this is mitigated by letting both target and domain of $\alpha_{C(\Gamma)}^{\shriek}$ sit as part of parametrised Thom spectra via the following:

\begin{definition}\label{Thomretracts}
We can construct the Thom space $C(\Gamma) \times \Fun(\coprod^k S^n,M)^{\RR^m(\complement \out(C(\Gamma)))}$ given by the homotopy Thom construction over the complement of outgoing colours $\coprod_{i=1}^k \complement N_i$ where we have extended along the trivial $\RR^m$-bundle over $\Emb(\coprod_ {i=1}^k \complement N_,Mi)$ for each set of outgoing colours in $C(\Gamma)$. The map $\alpha_{C(\Gamma)}^{\shriek}$ can be extended to a map into a suitable suspension 

$$\alpha_{C(\Gamma)+m}^{\shriek} \colon C(\Gamma)\times \Fun(\coprod^k S^n,M)^{\RR^m(\complement \out(C(\Gamma))} \to \Sigma_{\Emb(\coprod_{i=1}^k \out_i,M)}^m\Fun((S^n,\beta^{\thick}),M)^{TM(C(\Gamma))}$$ 

Under the map $\alpha_{[T,\underline{P}]}$, a point $x \in \complement N_i$ is mapped to a point $\alpha_{[T,\underline{P}]}(x) \in \beta_{[T,\underline{P}]}$, and we extend the $\RR^m$-suspension over $f(x)$ for $f \in \Emb(\coprod^k S^n,M)$ to the additional suspension over $\alpha^{\shriek}(f)(\alpha_{[T,\underline{P}]}(x))$.
\end{definition}

Extending to $C(\Gamma) \times \Fun(\coprod^k S^n,M)^{\RR^m(\complement \out(C(\Gamma))}$, we realize the maps into the space $\Fun((\coprod_{i=1}^k \out_i)_{C(\Gamma)},M)$ in (\ref{overdiagram}) as retracts by extending the maps to the point at $\infty$ away from the outgoing timber. The maps $\alpha_{C(\Gamma)+m}$ hence constitute a morphism of parametrised spectra.

The basic morphisms in the colimit $\colim \mathfrak{C} \cong \Cleav_S^n(-;k)$ are correspondences $$\xymatrix{C(\Gamma) & C(\Gamma,\Gamma+1)\ar[r]^{\phi_{\evade}}\ar[l]^{\phi_{\inc}} & C(\Gamma+1)}.$$

We can formula effect of these map on the associated action maps: 

\begin{proposition}
Letting $\alpha^{\shriek}_{C(\Gamma)},\alpha^{\shriek}_{C(\Gamma,\Gamma+1)}$ and $\alpha^{\shriek}_{C(\Gamma+1)}$ be considered as morphisms in the category of fibered spectra over $\Fun((\coprod_{i=1}^k \out_i)_{C(\Gamma)},M)$, there are commutative digrams of parametrised Thom spectra:

We use the notation of (\ref{overdiagram}) to indicate a suitable desuspension of the fibered spectra described in \ref{Thomretracts}
\begin{eqnarray}\label{suspensiondiagram}
\xymatrix@C=5em{
\Fun(\coprod^k S^n,M) \times C(\Gamma) \ar[r]^{\alpha^{\shriek}_{C(\Gamma)}} & \Fun((S^n,\beta^{\thick}),M)^{TM(C(\Gamma))}\\
\Fun(\coprod^k S^n,M) \times C(\Gamma+1,\Gamma) \ar[r]^{\alpha_{C(\Gamma+1,\Gamma)}}\ar[u]\ar[d]^{\Sigma^{\dim(M)}} & \Fun((S^n,\beta^{\thick}),M)^{TM(C(\Gamma+1,\Gamma))}\ar[u]\ar[d]^{\Sigma^{\dim(M)}}\\
\Fun(\coprod^k S^n,M)^{\RR^{\dim(M)}(\complement \out(C(\Gamma + 1)))} \times C(\Gamma +1 ) \ar[r]^{\alpha_{C(\Gamma+1)}^{\shriek}} & \Fun((S^n,\beta^{\thick}),M)^{TM(C(\Gamma+1))}
}
\end{eqnarray}

In this diagram, to simplify notation, we let $\Sigma$ denote the suspension fibered over $\Fun((\coprod_{i=1}^k \out_i)_{C(\Gamma)},M)$.

\end{proposition}
\begin{proof}
The blueprint under the map $\phi_{\evade}$ is such that the blueprint in the target $C(\Gamma + 1)$ has an extra component compared to the domain $C(\Gamma,\Gamma+1)$.

There is an isomorphism $$\Fun((S^n,\beta^{\thick}),M)^{TM(C(\Gamma+1,\Gamma))} \cong \Sigma_{\Fun((\coprod_{i=1} \out_i,M))}^{\dim(M)} \Fun((S^n,\beta^{\thick}),M)^{\phi_{\evade}(TM(C(\Gamma+1,\Gamma))}$$

This follows easily, for instance by extending the standard suspension isomorphism of Thom spaces to the parametrised setting.


This means that adding an additional component to the blueprint of the homotopy Thom spaces are homeomorphic to performing a $\dim(M)$-fold fibered suspension, hence proving the commutativity of the lower diagram. Commutativity of the top diagram follows since $\phi_{\inc}$ is an inclusion.
\end{proof}

We shall use the notation $\Fun(\coprod^k S^n,M) \times C(\Gamma)$ for the parametrised spectrum with the $m$'th entry $\Fun(\coprod^k S^n,M)^{\RR^m(\complement C(\Gamma))}$, and $\Fun(\coprod^k S^n,M) \times \Cleav_{S^n}(-;k)$ the colimit of the left-hand side of the diagram (\ref{suspensiondiagram})

In the category of parametrised Thom spectra we can take the colimit of these diagrams, and hence produce the parametrised spectrum map in the following

\begin{theorem}\label{globalmap}

The colimit of the diagrams (\ref{suspensiondiagram}) provide a global map of parametrised spectra

\begin{eqnarray}\label{actionmap}
\Fun(\coprod^k S^n,M) \times \Cleav_{S^n}(-;k) \to \colim_{C(\Gamma)}\Fun((S^n,\beta^{\thick}),M)^{TM(C(\Gamma))}
\end{eqnarray}

This leads to an action on homology
$$
\h_p(\Fun(\coprod^k S^n,M)) \otimes \h_q(\Cleav_{S^n}(-;k)) \to \h_{p+q-\dim(M)(k-1)}(\Fun(S^n,M))
$$
\end{theorem}

\begin{remark}\label{operadicaction}
That the map in (\ref{actionmap}) is operadic in the sense that the action associated to the operadic composition $$\circ_i \colon \Cleav_{S^n}(-;k) \times_{\Ob(\Cleav_{S^n})} \Cleav_{S^n}(-;m) \to \Cleav_{S^n}(-;k+m-1),$$ will have the $m$-ary operation only affecting the domain of the $i$'th domain in the $k$-ary operation.

However, by (C) of \ref{shriekproperties}, the action is not directly an action of a $\Sigma_k$-operad. There is a sign-change associated to acting by $\Sigma_k$. The sign change can be recovered by breaking the permutation up into neighboring transpositions and using (C) of \ref{shriekproperties}. This 'sign-error' was first discovered in \cite{GodinHigher} 
\end{remark}
\begin{proof}
Having provided the spectrum map from above, we need to account for the action on homology map. As given in \cite[4.]{umkehrCohenKlein} and \cite[Chap. 20]{MaySigurdsson}, we obtain ordinary homology of the spaces in the spectrum from the parametrised spectrum by considering the morphism of homology parametrised spectra

$$
\xymatrix{
\h_{\bullet}\left(\Fun((\out_i,M)_{C(\Gamma)},M);C(\Gamma) \times \Fun(\coprod^k S^n,M)\right) \ar[d] \\
\h_{\bullet}\left(\Fun((\out_i,M)_{C(\Gamma)},M);\Fun((S^n,\beta^{\thick},M)^{TM(C(\Gamma))}\right)
}
$$

Applying the functor that quotients the parametrizing $\Fun((\out_i,M)_{C(\Gamma)},M)$ from the homology spectrum -- and smashing with the Eilenberg-Maclane spectrum and applying the homotopy groups to obtain homology.

We can compute the homology of the colimit by means of the homology of the homotopy Thom spaces in the sense that the diagram (\ref{cleavingpushout}) of section \ref{collidingandevadingsection}, gives us Mayer-Vietoris sequences of the homotopy Thom spaces over $C(\Gamma),C(\Gamma,\Gamma')$ and $C(\Gamma')$. Apllying the five-lemma iteratively on these sequences, we can compute the homology of $\colim_{C(\Gamma) \in \mathfrak{C}} \Fun((S^n,\beta^{\thick}),M)^{TM(C(\Gamma))}$ by computing the homology of the homotopy Thom spaces in the limit. The degree shift on the right-hand side follows from the Thom isomorphism applied to these Thom spaces. Notice that the diagram (\ref{suspensiondiagram}) is such that the dimension shift is constant along the limit.

We can hence calculate the shift where $[T,\underline{P}] \in \Cleav_{S^n}(-;k)$ consist of $k-1$ disjoint hyperplanes. Here it gives precisely a difference of $(k-1)\dim(M)$ between the two sides as stated in the theorem.

The homology computation hence comes down to the homology of $\Fun((S^n,\beta^{\thick}),M)$, and the restriction maps to the boundary of the thickened blueprint provides a map to $\Fun(S^n,M)$, where $S^n$ is the outgoing colours of $[T,\underline{P}]$ glued along the simplices of the thickened blueprint. In effect, the target is $\h_{p+q-\dim(M)(k-1)}(\Fun(S^n,M))$ as stated. 
\end{proof} 

\section{Recovering String Topology}

In the formula (\ref{umkehrformula}) we introduced the scaling factor $S_{ij}^b$. This is essential for ensuring that the umkehr map $\alpha^{\shriek}$ does not have self-intersections, and hence has a domain given as a Thom space over $\Emb(\coprod^k S^n,M)$. We define a homotopy from the umkehr map $\alpha^{\shriek} \colon \Emb(\coprod^k S^n,M) \times \Cleav_{S^n}(-;k) \to \colim_{C(\Gamma) \in \mathfrak{C}} \Emb((S^n,\beta^{\thick}),M)^{TM(C(\Gamma))}$ to a less restricted umkehr map, that lands in mapping spaces. 

$$
\hat{\alpha}^{\shriek}_t \colon \Emb(\coprod^k S^n,M) \times \Cleav_{S^n}(-;k) \to \colim_{C(\Gamma) \in \mathfrak{C}} \Map(S^n,M)^{TM(C(\Gamma))}
$$

This map is defined as $\alpha^{\shriek}$, but with a homotopy of the scaling factor $S_{ij}^b$ in (\ref{scalingfactor}) via the following:

$$
[0,1] \ni t \mapsto \frac{\dist(\gamma(x_i),\gamma(x_j))}{\varepsilon\left((1-t)\left(\inf_{y \in \coprod^k S^n} \delta_y\right) + t\right)} := S_{ij}^b(t).
$$

Where $S_{ij}^b(0) = S_{ij}^b$, meaning that $\hat{\alpha^{\shriek}}_0 = \alpha^{\shriek}$. At $t=1$ we have $S_{ij}^b(1)$ which gives the modified umkehr map $\hat{\alpha}^{\shriek} := \hat{\alpha}^{\shriek}_t$. To $[T,\underline{P}] \in \Cleav_{S^n}(-;k)$ with outgoing colour $N_1,\ldots,N_k$, the modified umkehr map will have $\hat{\alpha}^{\shriek}(\complement N_i)$ indpendent of where $\hat{\alpha}^{\shriek}(N_j)$ is as a subset of $M$. This means that there will be a potential intersection of loops, so $\hat{\alpha^{\shriek}}$ will land in the Thom space over $\Map(S^n,M)$ which is defined completely analogously to $\Emb((S^n,\beta^{\thick}),M)^{TM(C(\Gamma))}$, with the fibered Thom spectrum given in the colimit being fibered over $\Map(\coprod_{i=1}^k N_i,M)$.

We use this deformed umkehr map from embedding spaces to extend it to a map from mapping spaces, as is the case for ordinary string topology:

$$\alpha^{\shriek \shriek} \colon \Map(\coprod^k S^n,M) \to \colim_{C(\Gamma) \in \mathfrak{C}} \Map((S^n,\beta^{\thick}),M)^{\widehat{TM(C(\Gamma))}}$$

To describe the target, consider $f \in \Map(\coprod^k S^n,M)$, and take the subset $U_f \subset \coprod^k S^n$ where $f(x) = f(y)$ for $x,y$ in different components of $\coprod_{i=1}^k \complement N_i$, where $N_1,\ldots,N_k$ are the outgoing colours of $[T,\underline{P}] \in \Cleav_{S^n}(-;k)$.

The map $\alpha^{\thick}_{[T,\underline{P}]}$ maps the subset $U_f$ onto the blueprint $\beta_{[T,\underline{P}]}^{\thick}$, and we can form the homotopy Thom spaces 

$$
\Map((S^n,\beta^{\thick}_{[T,\underline{P}]} \setminus \alpha_{[T,\underline{P}]}^{\thick}(U_f)),(M,TM))/\Map((S^n,\beta^{\thick}_{[T,\underline{P}]} \setminus \alpha_{[T,\underline{P}]}^{\thick}(U_f)),(M,CTM))
$$

In this sense, when $u \in U_f$, we require $f(u) \in M$ considered as the zero-section of $TM$.

Let $s(\beta^{\thick}_{[T,\underline{P}]} \cap S^n)$ denote the space of open subspaces of $\beta^{\thick}_{[T,\underline{P}]} \cap S^n$ including itself and $\emptyset$, and endow this with the Vietoris topology. Extending $s(\beta^{\thick}_{[T,\underline{P}]})$ over $[T,\underline{P}] \in C(\Gamma)$ gives a topology to 
$$\Map(S^n,M)^{\widehat{TM(\Gamma)}} := \Map(S^n,M)^{\widehat{TM(\Gamma)}_{U_f}} \rtimes s(C(\Gamma)).$$

The map $\alpha^{\shriek \shriek}$ is given by noting that for self-intersections occurring only in the image of the same component of $\coprod^k S^n$, we can extend the map $\hat{\alpha}^{\shriek}$ directly to mapping spaces since the definition of the umkehr map $\hat{\alpha}^{\shriek}$ is only affected by different components. We extend to points of $U_f$ where it is not defined by letting $$\alpha^{\shriek \shriek}(f)(u) = f(u)$$ for $u,u' \in U_f$ points with $f(u) = f(u')$. 

What the above has produced for us can be summarized in the following diagram

\begin{eqnarray}\label{homotopydefect}
\xymatrix{
\Emb(\coprod^k S^n,M) \times \Cleav_{S^n}(-;k) \ar[r]^{\alpha^{\shriek}}\ar@{=}[d] & \colim_{C(\Gamma) \in \mathfrak{C}} \Map(S^n,M)^{TM(C(\Gamma))}\ar[d]\\
\Emb(\coprod^k S^n,M) \times \Cleav_{S^n}(-;k) \ar[r]^{\hat{\alpha}^{\shriek}}\ar[d] & \colim_{C(\Gamma) \in \mathfrak{C}} \Map(S^n,M)^{TM(C(\Gamma))}\ar[d]\\
\Map(\coprod^k S^n,M) \times \Cleav_{S^n}(-;k) \ar@{-->}[r]^{\alpha^{\shriek \shriek}} & \colim_{C(\Gamma)\in \mathfrak{C}} \Map(S^n,M)^{\widehat{TM(C(\Gamma))}}
}
\end{eqnarray}

By construction, the top square is commutative up to homotopy. The lower square is commutative, but the effect on homotopy in the lower square is harder to compute due to the nature of $\colim_{C(\Gamma) \in \mathfrak{C}}\Map(S^n,M)^{\widehat{TM(C(\Gamma))}}$, which is not a colimit of homotopy Thom Space. We do however have the following:

\begin{proposition}\label{dimension1nottoobad}
Specifying $\Map(-,-)$ as the bifunctor of smooth maps, the space $\Map(S^1,M)^{\widehat{TM}(C(\Gamma))}$ is homotopy equivalent to $\Map(S^1,M)^{TM(C(\Gamma))}$
\end{proposition} 
\begin{proof}
Fix $[T,\underline{P}] \in C(\Gamma) \subset \Cleav_{S^1}(-;k)$. This defines the subset of complements of the outgoing colours: $\complement N_1,\ldots,\complement N_k$. Given $f \in \Map(\coprod^k S^1,M)$, let $d$ denote the same metric as used in the definition of the umkehr map $\alpha^{\shriek}$, and let

$$
A^f_i := \{x \in \complement N_i \mid \exists j \neq i, \exists y \in \alpha_{[T,\underline{P}]}^{-1}\left(\alpha_{[T,\underline{P}]}(x)\right) \cap \complement N_j, f(y)=f(x)\}
$$

Since $f$ is a smooth map, there is a tubular neighborhood $t(A_i^f) \supset A_i^f$ inside $\complement N_i$ such that in this neighborhood, the $d(f(y),f(x)) < \varepsilon$ for $x \in \complement N_i$ and $y \in \alpha^{-1}(\alpha(x)) \cap \complement N_j$. 

We impose an equivalence relation on $\Map(\coprod^k S^1,M)$ by letting $f' \sim f$ when $f'(x) = f(x)$ for $x \in \bigcup_{i=1}^k t(A_i^f)$. 

Imposing this equivalence relation for any $f \in \Map(\coprod^k S^1,M)$ yields a quotient space $\Map(\coprod^k S^1,M)/\sim$, which we can identify as having the subspaces $A^f_i$ of the domain collapsed to a single point. In effect, we can identify $\Map(\coprod^k S^1,M)/\sim$ as the space of maps from $\coprod^k S^1 \to M$ where, restricted to $\coprod_{i=1}^k \complement N_i$, the components only diverge by more than $\varepsilon$ at isolated points.

Since the spaces $A^f_i$ are contractible subsets of $S^1$, the quotient map $\Map(\coprod^k S^1,M) \to \Map(\coprod^k S^1,M)/\sim$ is a homotopy equivalence.

We define the homotopy Thom space of the quotient $\Map(\coprod^k S^1,M)/\sim$ by noting that $\alpha_{[T,\underline{P}]}\left(\bigcup_{i=1}^k \complement t(A_i^f)\right) \subset \beta_{[T,\underline{P}]}$ is defined such that this is precisely the region where $f$ is further than $\varepsilon$ away from self-intersecting, this defines similarly a quotient space on the relative mapping space $\Map((S^n,\beta^{\thick}),(M,TM))$, where we identify maps that are further than $\varepsilon$ away at $\alpha\left(\bigcap_{i=1}\complement t(A_i^f)\right)$. This allows us to define the relative mapping space 

\begin{eqnarray*}
\left(\Map((S^1,\beta^{\thick}),M)/\sim\right)^{TM(C(\Gamma))} := \\
\left(\Map((S^1,\beta^{\thick})_{C(\Gamma)},(M,TM))/\sim\right)/\Map((S^n,\beta^{\thick})_{C(\Gamma)},(M,CTM))/\sim
\end{eqnarray*}

From $\Map(\coprod^k S^1,M)/\sim$ we can construct an umkehr map that fits into a commutative diagram as follows:

$$
\xymatrix{
\Map(\coprod^k S^1,M)\ar[r]^(.35){\alpha^{\shriek \shriek}}\ar[d]^{\sim} & \Map((S^1,\beta^{\thick}),M)^{\widehat{TM(C(\Gamma))}}\ar[d]\\
\Map(\coprod^k S^1,M)/\sim\ar[r]^(.4){\alpha^{\shriek \shriek \shriek}} & (\Map((S^1,\beta^{\thick}),M)/\sim)^{TM(C(\Gamma))}
}
$$

We define the map $\alpha^{\shriek \shriek \shriek}$ by noting that under the equivalence relation $\sim$, we can take representatives of the mappings $f \colon \coprod^k S^1 \to M$ such that restricted to $\complement N_1\coprod \cdots \coprod \complement N_k$ they are within $\varepsilon$ of each other. We define $\alpha^{\shriek \shriek \shriek}(x) = \alpha^{\shriek \shriek}(x)$ for $x \notin t(A_i^f)$. In particular, when $t(A_i^f)= \emptyset$, $f$ is mapped to the point at $\infty$. Whenever $t(A_i^f) \neq \emptyset$ we can fix a bump-function and use this to parametrize along the deformation retraction from $t(A_i^f)$ to $A_i^f$ to extend $\alpha^{\shriek \shriek}$, such that $f$ are mapped to the $0$-section in $TM$ along $A_i^f$. Hence providing a continuous map into $\Map((S^1,\beta^{\thick}),M)/\sim^{TM(C(\Gamma))}$.

That $\Map((S^1,\beta^{\thick}),M)/\sim^{TM(C(\Gamma))}$ and $\Map((S^1,\beta^{\thick}),M)^{TM(C(\Gamma))}$ are homotopy equivalent follows by contractibility of the Vietoris topology we have imposed, and since the quotient map $\Map((S^1,\beta^{\thick},M) \to \Map((S^1,\beta^{\thick}),M)/\sim$ is a homotopy equivalence.
\end{proof}

As a consequence of the above, we obtain the following morphism of string topology for embedding spaces into string topology for mapping spaces:

\begin{theorem}\label{bvalgebramor}
Taking homology of the above spectrum leads to morphism of BV-algebras
$$
\xymatrix{
\hh_*(\Emb(\coprod^k S^1,M) \times \Cleav_{S^1}(-;k))\ar[r]^{\alpha^{\shriek}_*}\ar[d] & \hh_{*-(k-1)\dim(M)}(\Emb(S^1,M))\ar[d]\\
\hh_*(\Map(\coprod^k S^1,M) \times \Cleav_{S^1}(-;k))\ar[r]^{\alpha^{\shriek \shriek \shriek}_*} & \hh_{*-(k-1)\dim(M)}(\Map(S^1,M))
}.
$$

Here, the lower morphism is the usual BV-algebra structure in string topology.

\end{theorem}
It follows from the remark \ref{operadicaction} that for $\sigma \in \Sigma_k$, the induced map of $\alpha_{[T,\underline{P}]}^{\shriek}$ should satisfy
$$
(\alpha^{\shriek}_*)((f_1,\ldots,f_k),(\sigma.[T,\underline{P}])) = \sgn(\sigma)(f_{\sigma(1)},\ldots,f_{\sigma(k)},[T,\underline{P}])
$$
We shall counter the $\sgn(\sigma)$ on the right-side of the equation by making the choices of signs in Thom isomorphisms be given by the sign of the permutation of $\Sigma_k$, and hence depend on the symmetric action as well. This makes the action in the theorem an operadic one.
\begin{proof}
In \cite[5.21]{TarjeCleavage}, it is shown that $\Cleav_{S^1}$ is an operad whose actions on homology provide BV-algebras. 

We can compute the homology of $\colim_{C(\Gamma) \in \mathfrak{C}} \Map(S^1,M)^{TM(C(\Gamma))}$ and the homology of $\colim_{C(\Gamma) \in \mathfrak{C}} \Emb(S^1,M)^{TM(C(\Gamma))}$ by computing the homology of the homotopy Thom spaces in the limit via the same method applied in \ref{globalmap}.

\end{proof}



\pagebreak
\bibliography{referencer}

\def\cprime{$'$}
\providecommand{\bysame}{\leavevmode\hbox to3em{\hrulefill}\thinspace}
\providecommand{\MR}{\relax\ifhmode\unskip\space\fi MR }
\providecommand{\MRhref}[2]{%
  \href{http://www.ams.org/mathscinet-getitem?mr=#1}{#2}
}
\providecommand{\href}[2]{#2}
\begin{thebibliography}{GTZ14}

\bibitem[Bar14]{TarjeCleavage}
Tarje Bargheer, \emph{The cleavage operad and string topology of higher
  dimension}, Trans. Amer. Math. Soc. \textbf{366} (2014), 4209--4241.

\bibitem[CJ02]{homotopystring}
Ralph~L. Cohen and John D.~S. Jones, \emph{A homotopy theoretic realization of
  string topology}, Math. Ann. \textbf{324} (2002), no.~4, 773--798.

\bibitem[CK09]{umkehrCohenKlein}
Ralph~L. Cohen and John~R. Klein, \emph{Umkehr maps}, Homology, Homotopy Appl.
  \textbf{11} (2009), no.~1, 17--33. \MR{2475820 (2010a:55006)}

\bibitem[CS99]{chassullivan}
Moira Chas and Dennis Sullivan, \emph{String topology}, 1999,
  arXiv.org:math/9911159.

\bibitem[God07]{GodinHigher}
Veronique Godin, \emph{Higher string topology operations}, 2007,
  arXiv.org:0711.4859.

\bibitem[GTZ14]{GinotTradlerZeinalian}
Gr{\'e}gory Ginot, Thomas Tradler, and Mahmoud Zeinalian, \emph{Higher
  {H}ochschild homology, topological chiral homology and factorization
  algebras}, Comm. Math. Phys. \textbf{326} (2014), no.~3, 635--686.
  \MR{3173402}

\bibitem[Hu06]{PoHuMapping}
Po~Hu, \emph{Higher string topology on general spaces}, Proc. London Math. Soc.
  (3) \textbf{93} (2006), no.~2, 515--544. \MR{2251161 (2007f:55007)}

\bibitem[Lei04]{LeinsterHigher}
Tom Leinster, \emph{Higher operads, higher categories}, London Mathematical
  Society Lecture Note Series, vol. 298, Cambridge University Press, Cambridge,
  2004. \MR{MR2094071 (2005h:18030)}

\bibitem[MS06]{MaySigurdsson}
J.~P. May and J.~Sigurdsson, \emph{Parametrized homotopy theory}, Mathematical
  Surveys and Monographs, vol. 132, American Mathematical Society, Providence,
  RI, 2006. \MR{2271789 (2007k:55012)}

\end{thebibliography}

\end{document}